\documentclass[12pt]{amsart}
\usepackage[foot]{amsaddr}

\usepackage{amssymb,amsthm,amsmath,amstext,amsxtra}
\usepackage{fullpage}
\usepackage{bm}       % allows bold italic letters in math mode
\usepackage{mathtools} % allows more extendible arrows
\usepackage[all]{xy}
\usepackage{booktabs}
\usepackage{hyperref}
\hypersetup{colorlinks=true,urlcolor=blue,citecolor=blue,linkcolor=blue}
\usepackage{enumerate}
\usepackage{ stmaryrd }
\usepackage{enumitem}
\usepackage{comment}
\usepackage{colonequals}
\usepackage{tikz}
\usepackage{float} % Needed to force the placement of figures
\usepackage{subcaption}
\usepackage{breqn}

\usepackage{eqparbox}
\usepackage[numbers]{natbib}

\makeatletter
\newcommand*{\da@rightarrow}{\mathchar"0\hexnumber@\symAMSa 4B }
\newcommand*{\da@leftarrow}{\mathchar"0\hexnumber@\symAMSa 4C }
\newcommand*{\xdashrightarrow}[2][]{%
  \mathrel{%
    \mathpalette{\da@xarrow{#1}{#2}{}\da@rightarrow{\,}{}}{}%
  }%
}
\newcommand{\xdashleftarrow}[2][]{%
  \mathrel{%
    \mathpalette{\da@xarrow{#1}{#2}\da@leftarrow{}{}{\,}}{}%
  }%
}
\newcommand*{\da@xarrow}[7]{%
  \sbox0{$\ifx#7\scriptstyle\scriptscriptstyle\else\scriptstyle\fi#5#1#6\m@th$}%
  \sbox2{$\ifx#7\scriptstyle\scriptscriptstyle\else\scriptstyle\fi#5#2#6\m@th$}%
  \sbox4{$#7\dabar@\m@th$}%
  \dimen@=\wd0 %
  \ifdim\wd2 >\dimen@
    \dimen@=\wd2 %   
  \fi
  \count@=2 %
  \def\da@bars{\dabar@\dabar@}%
  \@whiledim\count@\wd4<\dimen@\do{%
    \advance\count@\@ne
    \expandafter\def\expandafter\da@bars\expandafter{%
      \da@bars
      \dabar@ 
    }%
  }%  
  \mathrel{#3}%
  \mathrel{%   
    \mathop{\da@bars}\limits
    \ifx\\#1\\%
    \else
      _{\copy0}%
    \fi
    \ifx\\#2\\%
    \else
      ^{\copy2}%
    \fi
  }%   
  \mathrel{#4}%
}
\makeatother

\numberwithin{equation}{section}

\theoremstyle{plain}
\newtheorem*{theorema*}{Theorem A}
\newtheorem{theorem}[equation]{Theorem}

\newtheorem{proposition}[equation]{Proposition}
\newtheorem{lemma}[equation]{Lemma}

\newtheorem{conjecture}[equation]{Conjecture}

\theoremstyle{definition}
\newtheorem{definition}[equation]{Definition}
\newtheorem{example}[equation]{Example}

\theoremstyle{remark}
\newtheorem{remark}[equation]{Remark}

\newenvironment{enumalph}
{\begin{enumerate}}
{\end{enumerate}}

\newenvironment{enumroman}
{\begin{enumerate}}
{\end{enumerate}}

\newcommand{\Z}{\mathbb{Z}}
\newcommand{\Q}{\mathbb{Q}}

\newcommand{\F}{\mathbb{F}}

\let\C\relax
\newcommand{\C}{\mathbb{C}}

\newcommand{\PP}{\mathbb{P}}

\newcommand{\GA}{\mathbf{G}_A}
\newcommand{\T}{\mathbf{T}}

\newcommand{\defi}[1]{\textsf{#1}} 	% for defined terms

\DeclareMathOperator{\Jac}{Jac}

\DeclareMathOperator{\GL}{GL}

\DeclareMathOperator{\Nm}{Nm}

\DeclareMathOperator{\Galois}{Gal}
\newcommand{\Gal}[2]{\Galois(#1\,|\,#2)}
\DeclareMathOperator{\Aut}{Aut}

\DeclareMathOperator{\Frob}{Frob}

\DeclareMathOperator{\rk}{\mathrm{rk}}

\DeclareMathOperator{\End}{\mathrm{End}}

\DeclareMathOperator{\alg}{{al}}

\newcommand{\frakp}{\mathfrak{p}}
\newcommand{\frakq}{\mathfrak{q}}

\newcommand{\FAconn}{F_A^{\textup{\tiny{conn}}}}

\usepackage{xcolor}
\definecolor{darkred}{HTML}{CC1F1F}
\definecolor{green}{rgb}{.4,.7,.4}
\definecolor{blue}{rgb}{.2,.6,.75}
\definecolor{pastelb}{HTML}{3333FF}

\definecolor{pastelyellow}{rgb}{0.992157, 0.552941, 0.235294}
\definecolor{pastelorange}{rgb}{0.941176, 0.231373, 0.12549}
\definecolor{pastelred}{rgb}{0.741176, 0., 0.14902}
\definecolor{darkbrown}{rgb}{0.25098, 0., 0.0745098}

\usepackage{hyperref}
\hypersetup{colorlinks=true,linkcolor=blue,anchorcolor=blue,citecolor=blue}

\usepackage[normalem]{ulem}

\DeclareMathOperator{\Hom}{Hom}
\DeclareMathOperator{\jvGal}{Gal}
\DeclareMathOperator{\bfGL}{\mathbf{GL}}
\DeclareMathOperator{\bfG}{\mathbf{G}}

\newcommand{\setplaces}{S_{MT}}
\def\conn{{\operatorname{conn}}}

\def\D{\mathbf{D}}
\def\G{\mathbf{G}}
\def\T{\mathbf{T}}
\def\bbar#1{\setbox0=\hbox{$#1$}\dimen0=.2\ht0 \kern\dimen0
\overline{\kern-\dimen0 #1}}
\newcommand{\kbar}{k^{\alg}}

\newcommand{\calW}{\mathcal W}
\theoremstyle{definition}
\newtheorem{algm}[equation]{Algorithm}

\newenvironment{enumalgalph}
{\begin{enumerate}}
{\end{enumerate}}

\newenvironment{enumalg}
{\begin{enumerate}}
{\end{enumerate}}

\newenvironment{enumromanii}
{\begin{enumerate}}
{\end{enumerate}}

\begin{document}

\title[Frobenius central endomorphisms]{Identifying central endomorphisms of an abelian variety via Frobenius endomorphisms}

%\author[Edgar Costa]{Edgar Costa$^{*}$}
%\thanks{${}^*$Corresponding author}
\author[Edgar Costa]{Edgar Costa}
\address[Edgar Costa]{
  Department of Mathematics\\
  Massachusetts Institute of Technology\\
  Cambridge, MA 02139, USA\\
  ORCiD: 0000-0003-1367-7785}
\email{edgarc@mit.edu}
\urladdr{\url{https://edgarcosta.org}}

\author{Davide Lombardo}
\address[Davide Lombardo]{
  Dipartimento di Matematica\\
  Universit\`{a} di Pisa\\
  Largo Bruno Pontecorvo 5\\
  56127, Pisa, Italy\\
  ORCiD: 0000-0002-1069-3379}
\email{davide.lombardo@unipi.it}
\urladdr{\url{http://people.dm.unipi.it/lombardo/}}

\author{John Voight}
\address[John Voight]{
  Department of Mathematics\\
  Dartmouth College\\
  6188 Kemeny Hall\\
  Hanover, NH 03755, USA\\
  ORCiD: 0000-0001-7494-8732}
\email{jvoight@gmail.com}
\urladdr{\url{http://www.math.dartmouth.edu/~jvoight/}}

\begin{abstract}
  Assuming the Mumford--Tate conjecture, we show that the center of the endomorphism ring of an abelian variety defined over a number field can be recovered from an appropriate intersection of the fields obtained from its Frobenius endomorphisms.
  We then apply this result to exhibit a practical algorithm to compute this center.
\end{abstract}

\date{\today}

\maketitle

\section{Introduction}

Let $F$ be a number field with algebraic closure $F^{\alg}$.
Let $A$ be an abelian variety over $F$ and let $A^{\alg} \colonequals A \times_F F^{\alg}$ be its base change to $F^{\alg}$.
For a prime $\frakp$ of $F$ (i.e., a nonzero prime ideal of its ring of integers), we write $\F_\frakp$ for its residue field, and when $A$ has good reduction at $\frakp$ we let $A_{\frakp}$ denote the reduction of $A$ modulo $\frakp$.

In this article, we seek to recover the center of the geometric endomorphism algebra of $A$ from the action of the Frobenius endomorphisms on its reductions $A_\frakp$.
Our main result is the following theorem.

\begin{theorem}
  \label{thm:mainthm}
  Let $A$ be an abelian variety over a number field $F$ such that $A^{\alg}$ is isogenous to a power of a simple abelian variety.
  Let $B \colonequals \End(A^{\alg}) \otimes \Q$ be the geometric endomorphism algebra of $A$, let $L \colonequals Z(B)$ be its center, and let $m \in \Z_{\geq 1}$ be such that $m^2=\dim_L B$.
  Suppose that the Mumford--Tate conjecture (Conjecture \textup{\ref{conj:MT}}) for $A$ holds.  Then the following statements hold.

\begin{enumalph}
\item There exists a set $S$ of primes of $F$ of positive density such that for each $\frakp \in S$:
  \begin{enumromanii}
  \item $A$ has good reduction at $\frakp$, and the reduction $A_\frakp$ is isogenous (over $\mathbb{F}_p$) to the $m$th power of a geometrically simple abelian variety over $\F_\frakp$; and
  \item The $\Q$-algebra $M(\frakp) \colonequals Z(\End(A_{\frakp}) \otimes \Q)$ is a field, generated by the $\frakp$-Frobenius endomorphism, and there is an embedding $L=Z(B) \hookrightarrow M(\frakp)$ of number fields.
  \end{enumromanii}
\item For any $\frakq \in S$, and for all $\frakp \in S$ outside of a set of density $0$ (depending on $\frakq$), if $M'$ is a number field that embeds in $M(\frakq)$ and in $M(\frakp)$, then $M'$ embeds in $L$.
\end{enumalph}
\end{theorem}

Theorem~\ref{thm:mainthm} relies crucially on work of Zywina~\cite{zywina-14}.  By an explicit argument, the result was proven for $A$ an abelian surface by Lombardo~\cite[Theorem 6.10]{lombardo-18}.
This theorem may be thought of as a kind of local-global principle for the center of the endomorphism algebra: roughly speaking, the center of the geometric endomorphism algebra of $A$ is the largest number field that embeds in the center of the geometric endomorphism algebra in a relevant set of reductions over finite fields.

The set $S$ in Theorem~\ref{thm:mainthm} may be taken as in Definition~\ref{def:descriptionofset}.   If $m$ and a model for $A$ are given, then there is an effectively computable subset $S' \subseteq S$ with $S \smallsetminus S'$ finite  (see Lemma \ref{lem:effectcomputS}).  The value of $m$ is effectively computable given a model for $A$ (see Lemma \ref{lem:guessm}), and in fact it is easy in practice to guess $m$ (see Remark \ref{rmk:guessm}).

The primary motivation for this theorem is an algorithmic application.
A result of Costa--Mascot--Sijsling--Voight~\cite[Proposition 7.4.7]{costa-mascot-sijsling-voight-18} gives a conditional way to rigorously certify that a numerical calculation \cite[\S 2.2]{costa-mascot-sijsling-voight-18} of the endomorphism ring of a Jacobian is correct.  This result is conditioned on a hypothesis~\cite[Hypothesis 7.4.6]{costa-mascot-sijsling-voight-18} that is directly implied by Theorem~\ref{thm:mainthm}(b).
In this way, Theorem~\ref{thm:mainthm} allows us to determine a sharp upper bound on $\rk \End(A^{\alg})$, conditional on the Mumford--Tate conjecture holding for $A$, and thereby compute $\End(A^{\alg})$ in practice whenever the abelian variety $A$ is explicitly given as a Jacobian of a curve over $F$ or, more generally, as an isogeny factor of one (hence in principle all abelian varieties, see~e.g.\ Milne~\cite[\S~III-10]{milneAV}).

Going a bit further in this direction, we present here an alternative method to compute the center of the geometric endomorphism algebra of $A$ in Algorithm~\ref{alg:mainctr}, again conditional on the Mumford--Tate conjecture, using the notion of normic polynomials (see Section~\ref{sec:galthy}).  This algorithm has the advantage of avoiding potentially impractical field intersections suggested by Theorem~\ref{thm:mainthm}(b).

One expects to have correctly identified the center $L$ as in the conclusion of Theorem~\ref{thm:mainthm} after testing $O([F_{A}^{\textup{conn}}:F]^2)$ pairs of primes $\frakp,\frakq$, where $F_{A}^{\textup{conn}}$ is the smallest extension of $F$ for which all the $\ell$-adic monodromy groups associated to $A$ are connected---but Algorithm~\ref{alg:mainctr} does not compute the field $F_{A}^{\textup{conn}}$ directly.
In particular, we prove the correctness of Algorithm~\ref{alg:mainctr} without establishing if the density zero set of primes in Theorem~\ref{thm:mainthm}(b) can be computed effectively.
Finally, even without assuming the Mumford--Tate conjecture for $A$, Algorithm~\ref{alg:mainctr}
still yields an \emph{upper bound} on the center of the geometric endomorphism algebra of $A$---we just have no guarantee that this upper bound is sharp.

We conclude with a result refining Theorem \ref{thm:mainthm} to obtain another arithmetically interesting field attached to $A$, namely the splitting field of the Mumford--Tate group (see Section \ref{sec:splitred} for a precise definition).  Keeping notation as in Theorem \ref{thm:mainthm}, for $\frakp \in S$ let $N(\frakp)$ be a normal closure of the extension $M(\frakp) \supseteq \Q$ generated by the $\frakp$-Frobenius endomorphism.

\begin{theorem} \label{thm:MST}
  \label{thm:splittingfieldMT}
  Let $A$ be an abelian variety over a number field $F$ such that $A^{\alg}$ is isogenous to a power of a simple abelian variety, and suppose that the Mumford--Tate conjecture for $A$ holds.  Let $F_{\G_A}$ be the splitting field of the Mumford--Tate group $\G_A$ of $A$.  Then the following statements hold.
  \begin{enumalph}
  \item There exists a subset $\setplaces \subseteq S$,  of the same density as $S$,
   such that for each $\frakp \in \setplaces$, conditions \textup{(i)}--\textup{(ii)} of Theorem~\textup{\ref{thm:mainthm}(a)} hold and moreover:
  \begin{enumerate}
  \item[\textup{(iii)}] There is an embedding $F_{\G_A} \hookrightarrow N(\frakp)$.
  \end{enumerate}
  \item For any $\frakq \in \setplaces$, and for all $\frakp \in \setplaces$ outside of a set of density $0$ (depending on $\frakq$), we have $N(\frakq) \cap N(\frakp) \simeq F_{\G_A}$.
  \end{enumalph}
\end{theorem}

In Theorem \ref{thm:MST}, the intersection $N(\frakq) \cap N(\frakp)$ is well-defined up to isomorphism since both fields are normal extensions of $\Q$, and so this intersection can be computed without to resorting to normic polynomials.  

For further work in this direction, see also the recent paper of Zywina \cite{zywina-20}, giving an algorithmic approach to the computation of the Mumford--Tate group itself up to isomorphism using the data of Frobenius polynomials.

\subsection*{Organization}

This article is organized as follows.  In section \ref{sec:galthy} we set up some basic Galois theory.  Then in section \ref{sec:splitred} we review what is needed from work of Zywina \cite{zywina-14} and Costa--Mascot--Sijsling--Voight \cite{costa-mascot-sijsling-voight-18} and prove Theorem \ref{thm:mainthm}.  Then in section \ref{sec:splitMT} we prove Theorem \ref{thm:MST}.  We conclude in section \ref{sec:algm} with an algorithmic application (Algorithm~\ref{alg:mainctr}).

\subsection*{Acknowledgements}

The authors would like to thank Andrea Maffei for an enlightening discussion about the Steiner section for reductive groups, the anonymous referees for their critical feedback, Claus Fieker, Mark van Hoeij, Tommy Hofmann, and Jeroen Sijsling for pointers, and David Zywina for helpful guiding discussions.
Costa  was supported by a Simons Collaboration Grant (550033).
Voight was supported by an NSF CAREER Award (DMS-1151047) and a Simons Collaboration Grant (550029).

\section{Galois theory} \label{sec:galthy}

In this section, we relate field embeddings to normic factors of a minimal polynomial using some basic Galois theory: see also Kl\"uners \cite{MR1673595}, van Heoij--Kl\"uners--Novocin \cite[Definition 5]{vHKN}, and Szutkoski--van Hoeij \cite[Theorem 4]{SvH}.  Throughout this section, let $K$ be a field with separable closure $K^{\textup{sep}}$.  For a field homomorphism $v\colon K \hookrightarrow L$ and a polynomial $f(T)=\sum_i a_i T^i \in K[T]$, we define
\[ (vf)(T) \colonequals \textstyle{\sum_i} v(a_i) T^i \in L[T] \]
to be the polynomial obtained by applying $v$ to the coefficients of $f$.

\begin{definition}
  Let $L \supseteq K$ be a separable field extension of finite degree.  For a polynomial $f(T) \in L[T]$, define the \defi{norm} from $L$ to $K$ of $f(T)$ to be
  \[ \Nm_{L|K}(f(T)) \colonequals \prod_{v\colon L \hookrightarrow K^{\textup{sep}}} (vf)(T), \]
where the product runs over the $[L:K]$ distinct $K$-embeddings $L \hookrightarrow K^{\textup{sep}}$.
\end{definition}

Since $\Gal{K^{\textup{sep}}}{K}$ permutes the embeddings $L \hookrightarrow K^{\textup{sep}}$, by Galois theory we have $\Nm_{L|K}(f(T)) \in K[T]$.  Accordingly, we may also define the norm as the product over the embeddings $L \hookrightarrow N$ for any Galois extension $N \supseteq K$ that has at least one such embedding.

\begin{example}
  If $f(T) \in K[T]$ is monic, irreducible, and separable, and $L=K(a)$ is the field obtained by adjoining a root $a$ of $f(T)$, then $\Nm_{L|K}(T-a)=f(T)$.
\end{example}

\begin{proposition}
  \label{prop:galthy}
  Let $g(T) \in K[T]$ be monic, irreducible, and separable, and let $a \in K^{\textup{sep}}$ be a root of $g(T)$.
  Let $L \supseteq K$ be a finite separable extension and let $h(T) \in L[T]$ be monic.
  Then the following conditions are equivalent:
  \begin{enumroman}
  \item $g(T)=\Nm_{L|K} h(T)$;
  \item There exists a $K$-embedding $\sigma\colon L \hookrightarrow K(a)$ such that $(\sigma h)(T)$ is the minimal polynomial of $a$ over $\sigma(L)$; and
  \item $h(T)$ is an irreducible factor of $g(T)$ in $L[T]$ and $\deg g(T)=[L:K]\deg h(T)$.
  \end{enumroman}
  Moreover, if $h(T)$ satisfies these equivalent conditions, then $L$ is generated over $K$ by the coefficients of $h(T)$.
\end{proposition}

\begin{proof} Throughout, let $N$ be a splitting field of $g(T)$ over $K$.

  We start with (i) $\Rightarrow$ (ii).
  Suppose that $g(T)=\Nm_{L|K} h(T)$ with $h(T) \in L[T]$.
  We first claim that $h(T)$ is irreducible in $L[T]$: if $d(T) \mid h(T)$ with $d(T) \in L[T]$ monic of positive degree, then $\Nm_{L|K} d(T) \mid \Nm_{L|K} h(T) = g(T)$ with $\Nm_{L|K} d(T) \in K[T]$; but $g(T)$ is irreducible in $K[T]$, so equality holds; and then by comparison of degrees we conclude that $d(T)=h(T)$.
  Next, let $\{\sigma_i\}_i = \Hom_K(L,N)$.
  Since $g(a)=0$ and $g(T)=\prod_i (\sigma_i h)(T)$, there exists $i$ such that $(T-a) \mid (\sigma_i h)(T)$.  Since $h(T)$ is irreducible in $L$, we conclude $h'(T) \colonequals (\sigma_i h)(T)$ is irreducible in $L' \colonequals \sigma_i(L)$ and so $h'(T)$ is the minimal polynomial of $a$ over $L'$.  Thus
  \begin{equation} \label{eqn:degarg}
    \begin{aligned}
      [K(a) : K] &= \deg g(T) = \deg h(T) [L:K] =\deg h'(T)[L':K] \\
                 &= [L'(a):L'] [L':K] = [L'(a):K];
    \end{aligned}
  \end{equation}
  since $L'(a) \supseteq K(a)$, by \eqref{eqn:degarg} we have $L'(a)=K(a)$ so $L' \subseteq K(a)$, and we may take $\sigma=\sigma_i$ in (ii).

  We now prove (ii) $\Rightarrow$ (iii).
  Since $g(T)$ is the minimal polynomial of $a$ over $K$ and $(\sigma h)(T)$ is the minimal polynomial of $a$ over $\sigma(L)$ we have $(\sigma h)(T) \mid g(T)$ in $\sigma(L)[T]$ so $h(T) \mid g(T)$ since $\sigma$ is a $K$-embedding.  Moreover,
\begin{equation}
  [K(a) : K] = \deg g(T)=\deg (\sigma h) (T) [\sigma(L):K] =\deg h(T) [L:K].
\end{equation}

  To conclude, we show (iii) $\Rightarrow$ (i). We are given $h(T) \mid g(T)$, so every root of $h(T)$ is a root of $g(T)$. The field $N$ contains all roots of $g(T)$ hence all roots of $h(T)$.  Let $b \in N$ be a root of $h(T)$, hence also of $g(T)$; since $g(T) \in K[T]$ is irreducible we conclude $g(T)$ is the minimal polynomial of $b$ over $K$.
  Let $n(T) \colonequals \Nm_{L|K} h(T) \in K[T]$.
  Then $n(b)=0$, so $g(T) \mid n(T)$.\
  But $\deg n(T)=[L:K]\deg h(T)=\deg g(T)$, so $g(T)=n(T)$ since both are monic.

  For the final statement, we may suppose (ii) holds and identify $L$ with its image in $K(a)$ under $\sigma$.  Let $L' \subseteq L$ be the subfield of $L$ generated by the coefficients of $h(T)$; then $[K(a):L']=[K(a):L]=\deg h(T)$ since $h(T)$ is irreducible, so $L'=L$.
\end{proof}

\begin{definition}
  Let $M \supseteq K$ be a finite separable extension, and let $g(T) \in K[T]$ be monic.  We say a polynomial $h(T) \in M[T]$ is \defi{normic} for $g(T) \in K[T]$ over $M$ if all of the following conditions hold:
  \begin{enumroman}
  \item $h(T)$ is monic;
  \item $h(T) \mid g(T)$; and
  \item $g(T)=\Nm_{L|K} h(T)$, where $L \subseteq M$ is generated over $K$ by the coefficients of $h(T)$.
  \end{enumroman}
\end{definition}

\begin{example}
If $h_1(T)$ is normic for $g(T)$ over $M$, with $L_1 \subseteq M$ the subfield generated by the coefficients of $h_1(T)$, and $K \subseteq L_2 \subseteq L_1$, then $h_2(T) \colonequals \Nm_{L_1|L_2} h_1(T)$ is also normic for $g(T)$ over $M$.
\end{example}

\begin{remark}
If $h(T)$ is normic for $g(T)$ over $M=K(a)$ and further $(T-a) \mid h(T)$, then van Heoij--Kl\"uners--Novocin call $h(T)$ the \defi{subfield polynomial} of $L$ \cite[Definition 5]{vHKN}; they state a version of Proposition \ref{prop:galthy} in their setting \cite[Remark 6]{vHKN}.  More recently, Szutkoski--van Hoeij \cite[Theorem 4]{SvH} have developed further equivalent conditions for subfield polynomials.
We will soon find ourselves in a situation that would be a very simple case of these algorithms, so we will not need to employ these more advanced techniques.
\end{remark}

We apply the previous bit of Galois theory as follows.

\begin{proposition} \label{prop:maxgal}
  Let $g(T) \in K[T]$ be monic, irreducible, and separable.
  Let $M \supseteq K$ be a finite separable extension.
  Then the following statements hold.
  \begin{enumalph}
  \item The set of normic polynomials for $g(T)$ over $M$ is a nonempty, partially ordered set under divisibility.
  \item Let $h_1(T) \mid h_2(T)$ be normic polynomials for $g(T)$ over $M$, and let $L_1, L_2 \subseteq M$ be the subfields generated over $K$ by the coefficients of $h_1(T),h_2(T)$, respectively.  Then $L_2 \subseteq L_1$.
  \end{enumalph}
\end{proposition}

\begin{proof}
  For part (a), the set is nonempty by taking $h(T)=g(T)$ (and $L=K$), and divisibility clearly gives a partial ordering.

  Now part (b).  Let $N$ be a splitting field for $g(T)$ over $K$, let $G \colonequals \Gal{N}{K}$ and $H_i \colonequals \Gal{N}{L_i}$ for $i=1,2$.  Let $\sigma \in H_1 \backslash H_2$.  Since $h_2(T)$ is normic for $g(T)$ and $g(T)$ is separable, $(\sigma h_2)(T)$ is coprime to $h_2(T)$.  But since $\sigma \in H_1$, we have
  \[ h_1(T)=(\sigma h_1)(T) \mid (\sigma h_2)(T), \]
  a contradiction.  So $H_1 \subseteq H_2$ and by the Galois correspondence $L_2 \subseteq L_1$.
\end{proof}

\begin{remark}\label{rmk:irreduciblenormicfactors}
Proposition~\ref{prop:maxgal}(a) does not assure the existence of an irreducible normic factor over $M$.
  For example, let $g(T) \in \Q[T]$ have degree $4$ and Galois group $\operatorname{Gal}(g(T)) = S_4$.
  Let $N$ be the splitting field of $g(T)$ over $\Q$ and let $M$ be the subfield of $N$ of degree $6$ fixed by the subgroup $H=\langle (1\,2), (3\,4) \rangle < S_4$.
  The polynomial $g(T)$ factors over $M[T]$ as a product of two irreducible degree-2 polynomials.
  By Proposition~\ref{prop:galthy}(iii), we conclude that neither factor can be normic, as $M$ does not have an intermediate field of degree 2.
  Indeed, the field generated by the coefficients of either factor is $M$ itself.
\end{remark}

\begin{remark}
  In Proposition~\ref{prop:maxgal}(b), the converse need not hold.
  For example, suppose that $M \colonequals K(a) \supseteq K$ is Galois, where $a \in K^{\textup{sep}}$ is a root of $g(T)$.
  Then $g(T)$ splits in $M$ and any linear factor generates $M$.
\end{remark}

\section{Splitting of reductions of abelian varieties} \label{sec:splitred}

In this section, we set up some notation and describe some results from Zywina \cite{zywina-14} concerning splitting of reductions of abelian varieties.  See also Costa--Mascot--Sijsling--Voight \cite{costa-mascot-sijsling-voight-18} for a summary in an algorithmic context.

We begin with a bit of notation.  Let $F$ be a number field with algebraic closure $F^{\alg}$ and let $\jvGal_F \colonequals \Gal{F^{\alg}}{F}$.  Let $A$ be an abelian variety over $F$ of dimension $g$ and let $A^{\alg} \colonequals A \times_F F^{\alg}$ denote the base change of $A$ to $F^{\alg}$.  Suppose that $A^{\alg}$ is isogenous to a power of a simple abelian variety (over $F^{\alg}$---ultimately, in algorithmic applications we will reduce to this case \cite[Remark 7.4.10]{costa-mascot-sijsling-voight-18}).  We write $\End(A)$ for the ring of endomorphisms of $A$ defined over $F$ and $\End(A)_\Q \colonequals \End(A) \otimes_{\Z} \Q$; if $K \supseteq F$ is an extension, we will write $\End(A_K)$ for the ring of endomorphisms defined over $K$.  Let $B \colonequals \End(A^{\alg})_\Q$ be the geometric endomorphism algebra of $A$, and
let $L \colonequals Z(B)$ be the center of $B$.  Then $L$ is a number field and $B$ is a central simple algebra over $L$.
Let $m^2 \colonequals \dim_L B$ with $m \in \Z_{\geq 1}$, so that $\dim_\Q B=m^2[L:\Q]$.

For a prime $\frakp$ of $F$, write $\F_\frakp$ for its residue field and $q \colonequals \#\F_\frakp$, let
$\F_\frakp^{\alg}$ be the algebraic closure of $\F_\frakp$, and let $\Frob_\frakp$ be the Frobenius automorphism of $\F_\frakp^{\alg}$ fixing $\F_\frakp$.  For $\frakp$ a prime of good reduction for $A$, write $A_\frakp$ for the reduction of $A$ over the residue field $\F_\frakp$ and $A_\frakp^{\alg}$ for the base change of $A_\frakp$ to $\F_\frakp^{\alg}$.

Let $\ell$ be a prime number.  Let $T_\ell A$ be the $\ell$-adic Tate module of $A$, a free $\Z_\ell$-module of rank $2g$.  Let $V_\ell A \colonequals T_\ell A \otimes_{\Z_\ell} \Q_\ell$; then there is a continuous homomorphism
\[ \rho_{A,\ell} \colon \jvGal_F \to \GL(V_\ell(A)) \simeq \GL_{2g}(\Q_\ell). \]
For a prime $\frakp$ of good reduction of $A$ that is coprime to $\ell$, let
\begin{equation} \label{eqn:cptis}
c_\frakp (T)  \colonequals \det(1 - \rho_{A,\ell}(\Frob_\frakp) T) \in 1+T\Z[T]
\end{equation}
be the inverse characteristic polynomial of the Frobenius $\Frob_\frakp$.  Then $c_\frakp(T)$ is independent of $\ell$.  Indeed, $c_\frakp(T)$ is the factor of the zeta function of $A_\frakp$ whose reciprocal roots have complex absolute value $\sqrt{q}$.  Thereby, $c_\frakp(T)$ can be recovered from the point counts $\#A(\F_{q^r})$ for $r=1,\dots,\max(2g, 18)$ \cite[\S 8]{kedlaya-quantum}; when $A\sim\Jac(C)$ is isogenous to the Jacobian of a curve $C$, one can also recover $c_{\frakp}(T)$ from $\#C(\F_{q^r})$ for $r=1,\ldots,g$.

Let $\bfGL(V_\ell(A))$ be the $\Q_\ell$-algebraic group of ($\Q_\ell$-linear) automorphisms of $V_\ell(A)$. The absolute Galois group $\jvGal_F$ acts by linear automorphisms on $V_\ell(A)$, hence we have $\rho_{A,\ell}(\jvGal_F) \leq \GL(V_\ell(A))= \bfGL(V_\ell(A))(\Q_\ell)$.  Let $\bfG_{A,\ell}$ be the Zariski closure of $\rho_{A,\ell}(\jvGal_F)$ in $\bfGL(V_\ell(A))$.  Then $\bfG_{A,\ell} \leq \bfGL(V_\ell(A))$ is an algebraic subgroup called the \defi{$\ell$-adic monodromy group} of $A$.  Let $\bfG_{A,\ell}^0$ be the identity component of $\bfG_{A,\ell}$.  Let $F_A^{\textup{conn}}$ be the fixed field in $F^{\alg}$ of $\rho_{A,\ell}^{-1}(\bfG_{A,\ell}^{0}(\Q_\ell))$.  Then $F_A^{\textup{conn}}$ is a finite Galois extension of $F$, independent of $\ell$ by a result of Serre \cite[p.\ 17]{serre-IV}.  The field $F_A^{\textup{conn}}$ is the smallest extension of $F$ for which the $\ell$-adic monodromy groups are connected for all primes $\ell$.

Choose an embedding $F \hookrightarrow \C$.  Let $V \colonequals H_1(A(\C),\Q)$; then $V_\C \colonequals V \otimes \C$ has a Hodge decomposition of type $\{(-1,0),(0,-1)\}$.  Let $\mu \, \colon \bfG_{m,\C} \to \bfGL(V_\C)$ be the cocharacter such that $\mu(z)$ acts as multiplication by $z$ on $V^{-1,0}$ and as the identity of $V^{0,-1}$ for all $z \in \C^\times=\bfG_{m,\C}(\C)$.  The \defi{Mumford--Tate group} of $A_\C$, denoted $\bfG_A$, is the smallest algebraic subgroup of $\bfGL(V)$ defined over $\Q$ such that $\bfG_A(\C)$ contains $\mu(\C^\times)$; then $\bfG_A$ is a reductive group over $\Q$ that is independent of the choice of embedding of $F$ into $\C$.

\begin{conjecture}[Mumford--Tate] \label{conj:MT}
The comparison isomorphism $V \otimes \Q_\ell \xrightarrow{\sim} V_\ell(A)$ identifies $\bfG_A \times_\Q \Q_\ell$ with $\bfG_{A,\ell}^0$.
\end{conjecture}

Let $\T \subset \GA$ be a maximal torus and $X(\T)$ be its character group. We write $\rk \G_A$ for the rank of $\G_A$ (i.e., the dimension of $\T$).
The \defi{absolute Weyl group} of $\GA$ with respect to $\T$ is the (finite) group
 \[
 W(\GA,\T) := N_{\GA}(\T)(\C)/\T(\C),
 \]
 where $N_{\GA}(\T)$ is the normalizer of $\T$ in $\GA$.  For an element $g\in N_{\GA}(\T)(\C)$, the map $\T_{\C}\to \T_{\C}$ given by $t\mapsto gtg^{-1}$ is an isomorphism of groups that depends only on the image of $g$ in $W(\GA,\T)$; this induces a faithful action of $W(\GA,\T)$ on $\T_{\C}$, so that we can identify $W(\GA,\T)$ with a subgroup of $\Aut(\T_{\C})$ and hence also of $\Aut(X(\T))$.

Any element $t \in \T(\C) \leq \GA(\C)$ acts on $V_\C$, and $\det(T-t)=\prod_{\alpha \in \Omega} (T-\alpha(t))^{m_\alpha}$, where $\Omega \subset X(\T)$ is the set of weights of $\T$ appearing in the representation $\GA \to \GL(V_\C)$ and the $m_\alpha$ are the corresponding multiplicities. As $W(\GA,\T)$ acts on $X(\T)$ stabilizing $\Omega$, in particular we obtain an action of $W(\GA,\T)$ on $\Omega$, hence on the roots of the characteristic polynomial of $t$.

We also recall the definition of the splitting field of $\G_A$.

\begin{definition}
The \defi{splitting field} of $\G_A$, denoted $F_{\G_A}$, is the intersection of all fields $K \subseteq \Q^{\alg}$ such that $\G_A \times_\Q K$ is split as a reductive group.
\end{definition}

The field $F_{\G_A}$ is a finite Galois extension of $\Q$.  With this notation in hand, we now introduce our set of primes.

\begin{definition} \label{def:descriptionofset}
Let $S$ be the set of primes $\frakp$ of $F$ with the following properties:
\begin{enumroman}
\item The prime $\frakp$ is a prime of good reduction for $A$;
\item $\Nm(\frakp)$ is prime, i.e., the residue field $\#\F_\frakp$ has prime
  cardinality;
\item $\End(A_\frakp^{\alg})$ is defined over $\F_\frakp$;
\item We have an isogeny $A_{\frakp} \sim Y_\frakp^{m}$ over $\F_\frakp$, with $Y_\frakp$ simple; and
\item The algebra $\End(Y_\frakp)_\Q$ is a field, generated by the Frobenius endomorphism.
\end{enumroman}
Let $\setplaces$ be the set of primes $\frakp$ satisfying (i)--(v) and
\begin{enumroman}
\item[(vi)] The roots of $c_{\frakp}(T)$ (defined in \eqref{eqn:cptis}) generate a free subgroup $\Phi_{\frakp} \leq (\Q^{\alg})^\times$ of rank equal to $\rk \G_A$.
\end{enumroman}
\end{definition}

We have $\setplaces \subseteq S$.  Given a model for $A$ (provided by equations in projective space), we consider the property:
\begin{enumroman}
\item[(i${}^{\prime}$)] The prime $\frakp$ is a prime of good reduction for the model of $A$.
\end{enumroman}
Let $S'$ be the set of primes satisfying (i${}^{\prime}$) and (ii)--(v) in Definition \ref{def:descriptionofset}.  The sets $S$ and $S'$ differ in only finitely many primes.  We define $\setplaces'$ similarly, satisfying (i${}^{\prime}$) and (ii)--(vi).

\begin{lemma} \label{lem:effectcomputS}
Given $m$ and a model for $A$, the set $S'$ is effectively computable.  If $\rk \G_A$ is also given, then $\setplaces'$ is effectively computable.
\end{lemma}

\begin{proof}
Condition (i${}^{\prime}$) can be checked by ensuring the model is smooth. We can check (ii) using standard algorithms, and we let $p \colonequals \#\F_\frakp$.  For such $\frakp$, we compute $c_\frakp(T)$ using a model of $A$ by counting $\#A(\F_{\frakp^r})$ as above.
(A finite list of primes containing those of bad reduction and the ability to compute $c_\frakp(T)$ for each good prime $\frakp$ are all we need from a model.) We can check conditions (iii), (iv), and (v) as follows.

For properties (iii)--(iv), we refer to Costa--Mascot--Sijsling--Voight \cite[Lemma 7.2.7]{costa-mascot-sijsling-voight-18} and Zywina~\cite[Lemma~2.1]{zywina-14} for details; we indicate only the key points here.  To verify (iii) we use the (proven) Tate conjecture: letting $c_\frakp^{\otimes 2}(T)$ be the characteristic polynomial of $\rho_{A,\ell}^{\otimes 2}(\Frob_\frakp)$, we verify that the only reciprocal roots of $c_\frakp^{\otimes 2}$ of the form $p\zeta$ with $\zeta$ a root of unity in fact have $\zeta=1$.  For (iv), we recall from Honda--Tate theory that an abelian variety over a prime finite field whose characteristic polynomial of Frobenius has no real roots is simple if and only if this polynomial is irreducible if and only if its endomorphism algebra is a field, generated by the Frobenius endomorphism.  With (iii) established, it follows that $c_\frakp(T)$ has no real roots: indeed, otherwise it would be divisible by $1 - p T^2$, but then $-p=(-\sqrt{p})(\sqrt{p})$ would be a root of $c_\frakp^{\otimes 2}(T)$.  We then verify (iv) and (v) by checking that $c_\frakp(T) \in \Q[T]$ is the $m$th power of an irreducible polynomial (in $\Q[T]$) --- this is where we use $m$.

To conclude, we claim that condition (vi) can be checked effectively if $\operatorname{rk} \G_A$ is known.  Let $N$ be a splitting field for $c_\frakp$; then the reciprocal roots of $c_\frakp$ are algebraic integers that are $p$-units in $N$, i.e., their valuation at any prime that does not lie above $p$ is $0$.  The unit group $\Z_{N}[1/p]^\times$ is a finitely generated abelian group.  Moreover, there is an effectively computable isomorphism from the set of elements of $N$ that belong to $\Z_{N}[1/p]^\times$ to an abstract finitely generated abelian group defined by a minimal set of generators and relations (see e.g.\ Cohen \cite[\S 7.4]{cohen:advanced}).  We then apply this isomorphism to the reciprocal roots of $c_\frakp(T)$, and by linear algebra over $\Z$ (Smith normal form) we compute a minimal presentation for the subgroup they generate and thereby check if this subgroup is free of the correct rank.
\end{proof}

\begin{lemma} \label{lem:guessm}
Suppose that the Mumford-Tate conjecture holds for $A$. Then given a model for $A$, the rank $\rk \G_A$ of the Mumford--Tate group and $m$ are effectively computable.  \end{lemma}

\begin{proof}
The algorithm runs using a day-and-night strategy.

By day, we pick up from Lemma \ref{lem:effectcomputS}.  In showing that $S'_{MT}$ is effectively computable, we showed that $\rk \Phi_\frakp$ is effectively computable for $\frakp \in S'_{MT}$; then necessarily $\rk \Phi_\frakp \leq \rk \G_A$.  On the assumption of the Mumford--Tate conjecture for $A$, this lower bound is sharp for a set of primes $\frakp$ of positive density \cite[Proposition 2.4]{zywina-14}.  In a similar way, we may obtain an upper bound for $m$: we have $m \leq m_\frakp$ for $\frakp \in S'$ if $c_\frakp(T)$ is an $m_\frakp$th power of an irreducible element of $\Q[T]$.  On Mumford--Tate, this upper bound is sharp for a set of primes $\frakp$ of positive density \cite[Theorem 1.2]{zywina-14}.  (Again, we only need access to the polynomials $c_{\mathfrak{p}}(T)$ for these bounds, not the model.)

By night, we complement these bounds by a (hopelessly slow) search for nontrivial algebraic cycles in powers of $A$. Since every algebraic cycle is an eigenvector for the action of the Mumford--Tate group on homology (see for example Deligne \cite[Article I, Proposition 3.4]{DMOS} or van Geemen \cite[Theorem 3.5]{van-Geemen-KS}), this gives an (eventually sharp) upper bound on the rank of $\G_A$. Similarly, one can also search for endomorphisms of $A$ again represented as algebraic cycles, which eventually gives a sharp lower bound for $m$.
 The algorithm halts when the lower and upper bounds for $\rk \G_A$ and $m$ meet, which will happen eventually (under the hypothesis of the Mumford--Tate conjecture).
\end{proof}

\begin{remark} \label{rmk:guessm}
The upper bound for $m$ in Lemma \ref{lem:guessm} only needs the characteristic polynomial $c_{\mathfrak{p}}(T)$ of Frobenius for primes $\frakp \in S'$; this upper bound is unconditional.  The upper bound is tight if the Mumford--Tate conjecture holds for $A$, and then in practice one can quickly guess $m$.
\end{remark}

We now record two important properties about primes in $S,\setplaces$.

\begin{proposition} \label{prop:containsL}
The following statements hold.
\begin{enumalph}
\item For all $\frakp \in S$, there exists a unique monic irreducible $g_\frakp(T) \in \Q[T]$ such that
    $$c_\frakp(T) = g_\frakp(T)^m.$$
\item Let $\frakp \in S$ and let $M \colonequals \Q[T]/(g_\frakp(T))$.  Then there exists an embedding $L \hookrightarrow M$.
\item For all primes $\frakp \in S$, there exists an irreducible $h_\frakp(T) \in L[T]$ such that
    $$ g_\frakp(T) = \Nm_{L|\Q}h_\frakp (T) $$
    and such that the coefficients of $h_\frakp(T)$ generate $L$ (over $\Q$).
\item Suppose that the Mumford--Tate conjecture for $A$ holds.  Then the sets $S,\setplaces$ have positive density, equal to $[F_A^{\conn}:F]^{-1}$.
\end{enumalph}
\end{proposition}

\begin{proof}
Part (a) was proven in Lemma \ref{lem:effectcomputS} (following from property (iv)).  Part (b), that the center embeds in each Frobenius field, follows from the (proven) Tate conjecture \cite[Corollary 7.4.4]{costa-mascot-sijsling-voight-18}.
For part (c), using part (b) we have an embedding $L \hookrightarrow \Q[T]/(g_\frakp(T))$, so $g_\frakp(T)$ is normic over $L$ by Proposition~\ref{prop:galthy} applied to the monic reciprocal polynomial $T^d g_\frakp(1/T) \in \Q[T]$, where $d=\deg g_\frakp(T)$.

Finally, part (d) is a slight refinement of fundamental work of Zywina~\cite{zywina-14}: the proof of \cite[Proposition 7.3.25]{costa-mascot-sijsling-voight-18} gives the result for $S$, and the statement for $\setplaces$ then follows using the fact that the set of primes satisfying (vi) has full density when $F=F_A^{\conn}$ \cite[Proposition 2.4(ii)]{zywina-14}.
\end{proof}

\begin{proposition} \label{prop:zywinaS}
Let $\frakq \in S$ and let $M \colonequals \Q[T]/(g_\frakq(T))$.
Suppose that the Mumford--Tate conjecture holds for $A$.  Then there exists an embedding $L \hookrightarrow M$, and an extension $N \supseteq M$, normal over $\Q$, such that for all $\frakp \in S$ outside of a set of density zero (depending on $\frakq$), the following hold:
\begin{enumalph}
\item The polynomial $g_\frakp(T)$ factors over $N[T]$ into exactly $[L:\Q]$ irreducible factors conjugate under $\Gal{N}{\Q}$.
\item Any such irreducible factor is normic for $g_\frakp(T)$ over $N$, and the subfield of $N$ generated by its coefficients is conjugate to $L$ (over $\Q$).
\end{enumalph}
\end{proposition}

\begin{proof}
We prove part (a) relying on work of Zywina \cite{zywina-14} and comparing the action of Galois groups and the Weyl group.  Let $N \supseteq \Q$ be a finite normal extension containing the fields $\FAconn$, $M$, and $F_{\G_A}$.

Let $\frakp \in S$.  By Proposition~\ref{prop:containsL}(d), the sets $S$ and $\setplaces$ have the same density, so avoiding a set of density zero we may suppose $\frakp \in \setplaces$.  Further avoiding a zero-density set depending on $N$, the orbits of the natural action of $\Galois_N$ on the roots of $g_\frakp(T)$ are the same as the orbits the action of the absolute Weyl group $W(\GA, \T)$  \cite[Proposition 6.6]{zywina-14}.  (In fact, there is a natural homomorphism $\Galois_N \to \Aut(X(\T))$ depending on $\frakp$ whose image lies in $W(\GA,\T)$ and which is an isomorphism for all primes of $N$ above a prime $\frakp \in S$ outside a set of density zero.)

We claim that there are $[L:\Q]$ such orbits by making a second comparison to the action on the weights.  The group $W(\mathbf{G}_A,\mathbf{T})$ also acts on the set $\Omega$ of weights of $\T$.  By Zywina \cite[Lemma 6.1(ii)]{zywina-14}, this action has $[L:\Q]$ orbits.  Using property (vi) of $\setplaces$, there is an element $t_\frakp \in \T(\C)$ whose characteristic polynomial agrees with $c_\frakp(T)$: see Noot \cite[Theorem 1.8]{MR2472133} or Zywina \cite[§4.2]{zywina-14}.  On the assumption the Mumford--Tate conjecture for $A$, the set $\Omega$ is in natural bijection with the roots of $g_\frakp(T)$ (equivalently, of $c_\frakp(T)$) by \cite[Lemma 6.2]{zywina-14} applied to $t=t_\frakp$.  The claim follows.

From the claim, the polynomial $g_\frakp(T) \in \Q[T]$ factors into $[L:\Q]$ irreducible factors in $N[T]$.  Since $g_\frakp(T) \in \Q[T]$ is irreducible, the irreducible factors of $g_\frakp(T)$ in $N[T]$ are conjugate under $\Gal{N}{\Q}$, so these factors are distinct and of common degree $\deg g_\frakp(T)/[L:\Q]$, proving (a).

Next, part (b).
Let $\frakp$ be a prime not among the set of exceptions in the previous paragraph.
Let $h_\frakp'(T) \in N[T]$ be such an irreducible factor of $g_\frakp(T)$ and $L'$ the number field generated by its coefficients.
As $\Gal{N}{\Q}$ acts transitively on the $[L:\Q]$ irreducible factors of $g_\frakp(T)$ with stabilizer $\Gal{N}{L'}$, by the orbit-stabilizer lemma we have
\[ [L:\Q]=\frac{\#\Gal{N}{\Q}}{\#\Gal{N}{L'}}=[L':\Q]. \]
Therefore condition (iii) in Proposition~\ref{prop:galthy} is satisfied, so $h_\frakp'(T)$ is normic for $g_\frakp(T)$, which is to say, $g_\frakp(T)=\Nm_{L'|K} h'_\frakp(T)$.
Thus, the minimal degree of a normic factor for $g_\frakp(T)$ over $N$ is $\deg g_\frakp(T)/[L:\Q]$.

On the other hand, by Proposition~\ref{prop:containsL}(b), there exists an embedding $L \hookrightarrow M \subseteq N$.
Then by Proposition~\ref{prop:containsL}(c), there exists a normic factor $h_\frakp(T) \in L[T] \subseteq N[T]$ for $g_\frakp(T)$.
With Proposition~\ref{prop:galthy}(iii) again satisfied, we conclude $\deg h_\frakp(T)=\deg g_\frakp(T)/[L:\Q]$, so $h_\frakp(T)\in L[T] \subseteq N[T]$ achieves the minimal degree of a normic factor of $g_\frakp(T)$ over $N$.
It follows that $h_\frakp(T)$ is one of the irreducible factors of $g_\frakp(T)$ in $N[T]$, hence is conjugate to $h_\frakp'(T)$ in $N$.  The coefficients of $h_\frakp(T)$ generate $L$ (as a subfield of $N$), and so each $L'$ is isomorphic and therefore Galois conjugate to $L$ in $N$.
\end{proof}

We now prove our first theorem.

\begin{proof}[Proof of Theorem \textup{\ref{thm:mainthm}}]
Let $S$ be the set defined in Definition \ref{def:descriptionofset}.  The set $S$ has positive density by Proposition~\ref{prop:containsL}(d).  Properties (iii) and (iv) together imply that $Y_\frakp$ is geometrically simple; then properties (i), (iii), (iv), and (v) of $S$ and Proposition~\ref{prop:containsL}(b) give properties (i) and (ii) in the theorem.

We turn to the final statement of the theorem. Let $\frakq \in S$ be fixed, let $M \colonequals M(\frakq)$, let $N \supseteq M$ be as in Proposition~\ref{prop:zywinaS}, and let $\frakp$ be a prime not in the exceptional set in this proposition.
Let $K \subseteq M$ be a number field that embeds in $M(\frakp) \colonequals \Q[T]/(g_\frakp(T))$; we show $K$ embeds in the center $L$.
Let $\sigma \colon K \hookrightarrow M(\frakp)$ be an embedding and let $a \in M(\frakp)$ be a root of $g_{\frakp}(T)$.
Then, by Proposition~\ref{prop:galthy}, the minimal polynomial of $a$ over $\sigma(K)$ pulls back under $\sigma$ to a normic polynomial $h_{\frakp,K}(T) \in K[T]$ for $g_\frakp(T)$ over $M$ whose coefficients generate $K$.  On the other hand, by Proposition~\ref{prop:zywinaS}(b), there exists a normic factor of $g_\frakp(T)$ over $N$ that is \emph{irreducible} in $N[T]$ and whose coefficients generate $L$, so after conjugating there exists $L' \subseteq N$ conjugate to $L$ and $h_{\frakp,L'}(T) \in L'[T]$ normic for $g_\frakp(T)$ over $N$ such that $h_{\frakp,L'}(T) \mid h_{\frakp,K}(T)$.  Then by Proposition~\ref{prop:maxgal}(b), since the coefficients of $h_{\frakp,L'}(T)$ generate $L'$ we conclude that $K \subseteq L' \simeq L$.
\end{proof}

\section{The splitting field of the Mumford--Tate group} \label{sec:splitMT}
In this section we prove Theorem \ref{thm:splittingfieldMT}. We start with the following lemma on algebraic groups, which is similar in spirit to results of Jouve--Kowalski--Zywina \cite[Lemma 2.3]{jouve-kowalski-zywina-13}.

\begin{lemma}\label{lemma:SplittingFieldTorus}
Let $\G \leq \GL_{n,k}$ be a (linear) reductive group over a perfect field $k$, let $\T \leq \G$ be a maximal torus, and let $t \in \T(k^{\alg})$ be any element. Let $\calW_t$  be the set of eigenvalues of $t$ and let $L \colonequals k(\calW_t)$.  Let $\Phi_t$ be the subgroup of $(k^{\alg})^\times$ generated by $\calW_t$.

Suppose that $\Phi_t$ is a free abelian group of rank equal to the dimension of $\T$.  Then $L$ is a splitting field for $\T$.
\end{lemma}
\begin{proof}
Let $\D$ be the $k$-subgroup of $\G$ generated by $t$. As $t$ is contained in a torus, it is a semisimple element, and this implies that $\D$ is a group of multiplicative type (the identity component $\D^0_{\kbar}$ of $\D_{\kbar}$ is a torus).
By Borel \cite[\S 8.4]{borel-91}, we have that $L$ is a splitting field of $\D$, so it suffices to show that $\D=\T$. Clearly $\D^0 \leq \D \leq \T$, so it is enough to prove that $\D^0$ is a torus of the same dimension as $\T$.
The group $\Phi_t$ can be identified with the image of the group homomorphism
\begin{align*}
\gamma_\D \colon X(\D) &\to (\kbar)^{\times} \\
\chi &\mapsto \chi(t),
\end{align*}
where $X(\D)$ is the character group of $\D$. Notice that $X(\D)$ is an abelian group of finite type, but not necessarily free.
We obtain
\[
\dim \T = \rk \Phi_t = \rk \gamma_\D(X(\D)) \leq \rk X(\D) = \dim \D^0,
\]
which concludes the proof.
\end{proof}

\begin{lemma}\label{lemma:EveryFrobeniusFieldIsASplittingField}
Let $\frakp \in \setplaces$
and let $\mathcal{W}_\frakp \subseteq (\Q^{\alg})^\times$ be the set of roots of $g_{\frakp}(T)$.  Then the Mumford--Tate group $\G_A$ is split over the field $\mathbb{Q}(\mathcal{W}_\frakp)$.
\end{lemma}

\begin{proof}
Let $\T$ be a maximal torus of $\G_A$. As explained by Zywina \cite[\S 6.2]{zywina-14}, there exists $t_\frakp \in \T(\Q^{\alg})$ such that $c_{\frakp}(T)=\det(T-t_\frakp)$, so the eigenvalues of $t_{\frakp}$ are precisely the roots of $c_{\frakp}(T)$ (equivalently, of $g_{\frakp}(T)$). By definition of $\setplaces$, the group $\Phi_\frakp < (\Q^{\alg})^\times$ generated by $\mathcal{W}_\frakp$ is free of rank equal to the rank of $\G_A$, so we can apply Lemma \ref{lemma:SplittingFieldTorus}.
\end{proof}

We are now ready to prove Theorem \ref{thm:splittingfieldMT}.

\begin{proof}[Proof of Theorem \textup{\ref{thm:splittingfieldMT}}]
Let $\setplaces$ be the set of Definition \ref{def:descriptionofset}.
Since $\setplaces \subseteq S$, we have already shown property (i) in Theorem \ref{thm:mainthm}(a), and $\setplaces, S$ have the same density by Proposition \ref{prop:containsL}.  For $\mathfrak{p} \in \setplaces$, let $\mathcal{W}_{\frakp}$ the set of roots of $c_\frakp(T)$ in $(\Q^{\alg})^\times$, so $N(\frakp)=\Q(\mathcal{W}_\frakp)$.
By Lemma \ref{lemma:EveryFrobeniusFieldIsASplittingField}, for every $\frakp \in \setplaces$, the Mumford--Tate group $\G_A$ is split over $\mathbb{Q}(\mathcal{W}_{\frakq})$, which proves (a).

Suppose now that $F=F_A^{\textup{conn}}$.  Applying a result of Zywina \cite[Proposition 6.6]{zywina-14} (with $L=F_{\G_A}$), there is a set $\Sigma_1$ of primes of density zero such that for every $\mathfrak{p} \in \setplaces \smallsetminus \Sigma_1$, we have $\Gal{F_{\G_A}(\calW_{\frakp})}{F_{\G_A}} \simeq W(\GA,\T)$.
Let $\frakq \in \setplaces$.  Since $F_{\G_A} \subseteq \mathbb{Q}(\calW_{\frakq})$, by Lemma \ref{lemma:EveryFrobeniusFieldIsASplittingField}, we have $F_{\G_A}(\calW_{\frakq})=\mathbb{Q}(\calW_{\frakq})=N(\frakq)$. Applying the result of Zywina \cite[Proposition 6.6]{zywina-14} again
(now with $L=N(\frakq)$), there is a set $\Sigma_{2,\frakq}$ (depending on $\frakq$) of primes of density zero such that for every $\frakp \in \setplaces \smallsetminus (\Sigma_1 \cup \Sigma_{2,\frakq})$ we have $\Gal{N(\frakq)(\calW_{\frakp})}{N(\frakq)} \simeq W(\GA,\T)$ and $\Gal{F_{\G_A}(\calW_{\frakp})}{F_{\G_A}} \simeq W(\GA,\T)$.  This means precisely that the two fields $F_{\G_A}(\calW_{\frakp})=N(\frakp)$ and $N(\frakq)$ are linearly disjoint over $F_{\G_A}$, hence $N(\frakq) \cap N(\frakp)=F_{\G_A}$.  This proves (b) in the case $F=F_A^{\conn}$.

The general case follows by extension to $F_A^{\conn}$, taking the set of primes of $F$ that lie below the set of primes of $F_A^{\conn}$ constructed in the previous paragraph.
\end{proof}

\section{Algorithm}
\label{sec:algm}

In this section, we exhibit how Theorem \ref{thm:mainthm} can be used effectively to compute the center $L$ of a geometric endomorphism algebra.  We keep notation as introduced in section \ref{sec:splitred}.

\begin{algm} \label{alg:mainctr} \ \\
  Input:
  \begin{itemize}
      \item $m \in \Z_{\geq 1}$ such that $m^2=\dim_L B$,
      \item $C \in \Z_{\geq 1}$, and
      \item $c_\frakp(T) \in 1 + T\Z[T]$ as in \eqref{eqn:cptis} for all good primes $\frakp$ with $\Nm \frakp \leq C$.
      \end{itemize}
      Output:
      \begin{itemize}
          \item a boolean; if this boolean is \textsf{true}, then further
          \item $d_C \in \Z_{\geq 1}$ such that $[L:\Q] \leq d_C$, and
          \item $\{L_{C,i}\}_i$, a set of number fields such that for some $i$ there exists an embedding $L \hookrightarrow L_{C,i}$ of number fields.
      \end{itemize}
      Steps:
  \begin{enumalg}
  \item Using Lemma \ref{lem:effectcomputS}, compute the set of primes $S'_C \colonequals S' \cap \{\frakp : \Nm \frakp \leq C\}$.  If $S'_C=\emptyset$, return \textsf{false}.
  \item Choose $\frakq \in S'_C$ and initialize $M \colonequals \Q[T]/(g_{\frakq}(T))$ where $c_\frakq(T)=g_\frakq(T)^m$.
  \item For each prime $\frakp \in S'_C$ with $\frakp \neq \frakq$:
    \begin{enumalgalph}
    \item Let $g_\frakp(T) \in \Q[T]$ be such that $g_\frakp(T)^m=c_\frakp(T)$.
    \item Factor $g_\frakp(T)$ into irreducibles in $M[T]$.
    \item Compute the set of normic factors $h_{\frakp,i}(T) \mid g_\frakp(T)$ by checking condition (iii) of Proposition \ref{prop:galthy} for each divisor of $g_\frakp(T)$ (using the factorization in Step 2).  If no factor is normic, remove $\mathfrak{p}$ from the set $S'_C$ and continue with the next prime.
    \item For each normic divisor $h_{\frakp,i}(T)$, compute the subfield $L_{\frakp,i} \subseteq M$ generated over $\Q$ by its coefficients.
    \item Reduce $\{L_{\frakp,i}\}_i$ to a subset of representatives up to isomorphism of number fields.
    \item Let $d_\frakp \colonequals \max_i\,[L_{\frakp,i}:\Q]$ and let $r_\frakp \colonequals \#\{L_{\frakp,i} : [L_{\frakp,i}:\Q]=d_\frakp\}$.
    \end{enumalgalph}

\item If now $S'_C=\emptyset$, return \textsf{false}.

  \item Let $\frakp$ minimize first $\min_\frakp d_\frakp$ then $\min_\frakp r_\frakp$.
   For any such minimal prime $\frakp$, return \textsf{true}, $d_C \colonequals d_\frakp$ and the set of subfields $\{L_{\frakp,i} : [L_{\frakp,i}:\Q]=d_\frakp\} $.
  \end{enumalg}
\end{algm}

\begin{proof}[Proof of correctness]
  By Proposition~\ref{prop:containsL}(b), for each good $\frakp$ there is an embedding $\sigma\colon L \hookrightarrow \Q[T]/(g_\frakp(T))$.
  By finiteness, there exists a maximal subextension $\Q \subseteq L \subseteq L' \subseteq M$ with an embedding $L' \hookrightarrow \Q[T]/(g_\frakp(T))$, which we may take as extending $\sigma$.
  By (ii) $\Rightarrow$ (iii) of Proposition \ref{prop:galthy}, there exists a normic factor $h_{\frakp,i}(T) \mid g_\frakp(T)$ such that $L_{\frakp,i}=L'$.  Therefore the algorithm gives correct output for any prime $\frakp$ selected in Step 5.
\end{proof}

\begin{remark}
In step 3c we cannot limit ourselves to testing \textit{irreducible} factors, because a polynomial $f(T) \in \mathbb{Q}[T]$ may in general have no irreducible normic factors in $M[T]$, see Remark \ref{rmk:irreduciblenormicfactors}.
\end{remark}

\begin{proposition} \label{prop:algcorrect}
  Suppose that the Mumford--Tate conjecture for $A$ holds.  Then for large enough $C$, Algorithm \textup{\ref{alg:mainctr}} returns \textsf{true}, $d_C=[L:\Q]$, and a singleton  $\{L_{C,i}\}$, such that $[L_{C,i}:\Q]=d_C$ and $L \simeq L_{C,i}$.
\end{proposition}

\begin{proof}
  By Proposition \ref{prop:zywinaS}, there exists an embedding $L \hookrightarrow M$ and
  an extension $N \supseteq M$, with $N$ normal over $\Q$, such that $g_\frakp(T)$ factors over $N[T]$ with exactly $[L:\Q]$ irreducible factors for all $\frakp \in S$ outside a set of density zero.  Moreover, each such irreducible factor is normic, and the number field generated by its coefficients is conjugate to $L$ (over $\Q$) by an element of $\Gal{N}{\Q}$.
  For $C$ large enough, in the course of the algorithm we will eventually find $\frakp \in S' \subseteq S$ which is not in this density zero set of exceptions.

  We first claim that such a prime $\frakp$ does not get discarded in Step 3.  Indeed, let $h'_{\frakp}(T)$ be any irreducible factor of $g_\frakp(T)$ in $N[T]$.  Then $h'_{\frakp}(T)$ is normic.  Moreover, its field of coefficients is isomorphic to $L$, so after Galois conjugation we may suppose it is equal to $L$.  Then $h'_{\frakp}(T) \in L[T] \subseteq M[T]$, so it is an irreducible factor of $g_{\frakp}(T)$ in $M[T]$, and still normic, so $h'_{\frakp}(T)=h_{\frakp,i}$ for some $i$, passing Step 3c.

  Next, we claim that for such a prime $\frakp$ we have $d_\frakp = [L:\Q]$, $r_\frakp=1$, and $\{L_{\frakp,i}\}_i=\{L\}$ (up to isomorphism).  Indeed, let $h_{\frakp,j}(T)$ be another normic factor of $g_\frakp(T)$ from Step 3c with $i \neq j$ and field of coefficients $L_{\frakp,j}$.  From the earlier factorization of $g_\frakp(T)$ over $N$ into normic irreducibles, there exists a normic factor of $h_{\frakp,j}(T)$ over $N[T]$ whose field of coefficients is Galois conjugate to $L$.  By Proposition \ref{prop:maxgal}(b), we conclude that $L_{\frakp,j}$ is contained in this field of coefficients, thus $[L_{\frakp,j}:\Q] \leq [L:\Q] = [L_{\frakp,i}:\Q]$ with equality if and only if $L_{\frakp,j} \simeq L$ if and only if $\deg h_{\frakp,j}(T)=d_\frakp$.

  To conclude, suppose that $\frakp'$ is a prime such that $d_{\frakp'}=d_\frakp=[L:\Q]$ and $r_{\frakp'}=1$.
  Then $L \hookrightarrow L_{\frakp',1}$ so by degrees $L \simeq L_{\frakp',1}$ and the desired conclusion holds.
\end{proof}

\begin{remark}
  In particular, the quantity $r_\frakp$ in Algorithm \ref{alg:mainctr} is only used when $C$ is not yet large enough for Proposition \ref{prop:algcorrect} to apply, and is used only to possibly reduce the size of the output (without affecting its correctness).
  \end{remark}

\begin{example}
For a very simple example of the algorithm, consider the elliptic curve with LMFDB
  label
    \href{http://www.lmfdb.org/EllipticCurve/Q/11/a/2}{\textsf{11.a2}}
 a model for the modular curve $X_0(11)$.
One can easily verify that $2, 3 \in S'$ and that
 $M(2) \simeq \Q(\sqrt{-1})$ and $M(3) \simeq \Q(\sqrt{-11})$.
 Thus by~Theorem~\ref{thm:mainthm}, $L = \Q$ and therefore $\End E^{\alg} = \Z$.
\end{example}

\begin{example}
 Let $X$ be the curve defined in $\PP^3$ by the equations
  \begin{equation}
    \begin{aligned}
    -yz - 12 z^{2} + x w - 32 w^{2} & = 0
    \\
    y^{3} + 108 x^{2} z + 36 y^{2} z + 8208 x z^{2} - 6480 y z^{2} + 74304 z^{3} + 96 y^{2} w
\\
                \hspace{4em} + 2304 y z w - 248832 z^{2} w + 2928 y w^{2} - 75456 z w^{2} + 27584 w^{3} & = 0.
  \end{aligned}
\end{equation}
Then $X$ is a canonically embedded curve of genus $4$.  This curve arises in the classification of elliptic curves over $\Q$ with constrained $3$-adic Galois image (in upcoming work of Jeremy Rouse, Drew Sutherland, and David Zureick-Brown): it can obtained as the image of a $\Q$-rational basis of modular forms attached to the newspace with LMFDB label \href{http://www.lmfdb.org/ModularForm/GL2/Q/holomorphic/81/2/c/b/}{\textsf{81.2.c.b}} of level $81$.  We use this example as a test case for our algorithm, ignoring its modular provenance.  Let $J \colonequals \Jac(X)$ be the Jacobian of $X$.

With a Gr\"obner basis computation one can show that $X$ has good reduction away from $2$ and $3$, and hence also $J$.
%> P<x, y, z, w> := PolynomialRing(IntegerRing(), 4);
%> Iw := ideal<P | Minors(Matrix([[ Derivative(f, P.i) : i in [1..4]] : f in [f1,f2]]),2 ) cat [f1, f2] cat [w - 1]>;
%> P<x, y, z, w> := PolynomialRing(IntegerRing(), 4);
%> f1 := -y*z - 12*z^2 + x*w - 32*w^2;
%> f2 := y^3 + 108*x^2*z + 36*y^2*z + 8208*x*z^2 - 6480*y*z^2 + 74304*z^3 + 96*y^2*w +  2304*y*z*w - 248832*z^2*w + 2928*y*w^2 - 75456*z*w^2 + 27584*w^3;
%> Gs := [GroebnerBasis(ideal<P | Minors(Matrix([[ Derivative(f, P.i) : i in [1..4]] : f in [f1,f2]]),2 ) cat [f1, f2] cat [P.j - 1]>) : j in [1..4]];
%print Factorisation(LCM([G[#G] : G in Gs]));
%[
%    <2, 8>,
%    <3, 4>
%]
By point counting on the reduction of $X$ modulo $p$ one can compute $c_p(T)$, for $p \neq 2, 3$ which is feasible for small primes.
By employing Lemma~\ref{lem:guessm}, we guess $m = 4$ and
under that assumption we have that the first two primes in $S^{\prime}$ are $19$ and $37$.
Furthermore, we have
\begin{equation}
  \begin{aligned}
    %S^{\prime}_{181} =& \{19, 37, 73, 109, 127, 163, 181\} \\
    g_{19}(T) &=  1 - 2 T + 19 T^{2},  & M(19) &\simeq \Q(\sqrt{-2});
    \\
    g_{37}(T) &=  1 + 7 T + 37 T^{2}, & M(37) &\simeq \Q(\sqrt{-11}).
\end{aligned}
\end{equation}
We conclude that $L = \Q$.  In fact, we can indeed verify \cite{costa-mascot-sijsling-voight-18} that $J$ is of $\GL_2$-type over $\Q$,
and geometrically we have an isogeny $J^{\alg} \sim E^4$ where $E$ is an elliptic curve
whose $j$-invariant satisfies $j^2 - 7317j + 283593393=0$.
%[[19, 19T^{2} - 2T + 1],
% [37, 37T^{2} + 7T + 1],
% [73, 73T^{2} + 7T + 1],
% [109, 109T^{2} - 11T + 1],
% [127, 127T^{2} - 2T + 1],
% [163, 163T^{2} + 16T + 1],
% [181, 181T^{2} - 2T + 1]]
\end{example}

\bibliographystyle{alpha}
\bibliography{biblio}

\begin{thebibliography}{DMOS82}

\bibitem[Bor91]{borel-91}
Armand Borel.
\newblock {\em Linear algebraic groups}, volume 126 of {\em Graduate Texts in
  Mathematics}.
\newblock Springer-Verlag, New York, second edition, 1991.

\bibitem[CMSV19]{costa-mascot-sijsling-voight-18}
Edgar Costa, Nicolas Mascot, Jeroen Sijsling, and John Voight.
\newblock Rigorous computation of the endomorphism ring of a {J}acobian.
\newblock {\em Math. Comp.}, 88(317):1303--1339, 2019.

\bibitem[Coh00]{cohen:advanced}
Henri Cohen.
\newblock {\em Advanced topics in computational number theory}, volume 193 of
  {\em Graduate Texts in Mathematics}.
\newblock Springer-Verlag, New York, 2000.

\bibitem[DMOS82]{DMOS}
Pierre Deligne, James~S. Milne, Arthur Ogus, and Kuang-yen Shih.
\newblock {\em Hodge cycles, motives, and {S}himura varieties}, volume 900 of
  {\em Lecture Notes in Mathematics}.
\newblock Springer-Verlag, Berlin-New York, 1982.

\bibitem[JKZ13]{jouve-kowalski-zywina-13}
Florent Jouve, Emmanuel Kowalski, and David Zywina.
\newblock Splitting fields of characteristic polynomials of random elements in
  arithmetic groups.
\newblock {\em Israel J. Math.}, 193(1):263--307, 2013.

\bibitem[Ked06]{kedlaya-quantum}
Kiran~S. Kedlaya.
\newblock Quantum computation of zeta functions of curves.
\newblock {\em Comput. Complexity}, 15(1):1--19, 2006.

\bibitem[Kl{\"u}99]{MR1673595}
J{\"u}rgen Kl{\"u}ners.
\newblock On polynomial decompositions.
\newblock {\em J. Symbolic Comput.}, 27(3):261--269, 1999.

\bibitem[Lom19]{lombardo-18}
Davide Lombardo.
\newblock Computing the geometric endomorphism ring of a genus-2 {J}acobian.
\newblock {\em Math. Comp.}, 88(316):889--929, 2019.

\bibitem[Mil08]{milneAV}
James~S. Milne.
\newblock Abelian varieties (v2.00), 2008.
\newblock Available at www.jmilne.org/math/.

\bibitem[Noo09]{MR2472133}
Rutger Noot.
\newblock Classe de conjugaison du {F}robenius d'une vari\'{e}t\'{e}
  ab\'{e}lienne sur un corps de nombres.
\newblock {\em J. Lond. Math. Soc. (2)}, 79(1):53--71, 2009.

\bibitem[Ser13]{serre-IV}
Jean-Pierre Serre.
\newblock {\em Oeuvres/{C}ollected papers. {IV}. 1985--1998}.
\newblock Springer Collected Works in Mathematics. Springer, Heidelberg, 2013.
\newblock Reprint of the 2000 edition.

\bibitem[SvH17]{SvH}
Jonas Szutkoski and Mark van Hoeij.
\newblock The complexity of computing all subfields of an algebraic number
  field.
\newblock {\em preprint}, 2017.
\newblock {\tt arXiv:1606.01140}.

\bibitem[vG00]{van-Geemen-KS}
Bert van Geemen.
\newblock Kuga-{S}atake varieties and the {H}odge conjecture.
\newblock In {\em The arithmetic and geometry of algebraic cycles ({B}anff,
  {AB}, 1998)}, volume 548 of {\em NATO Sci. Ser. C Math. Phys. Sci.}, pages
  51--82. Kluwer Acad. Publ., Dordrecht, 2000.

\bibitem[vHKN13]{vHKN}
Mark van Hoeij, J\"{u}rgen Kl\"{u}ners, and Andrew Novocin.
\newblock Generating subfields.
\newblock {\em J. Symbolic Comput.}, 52:17--34, 2013.

\bibitem[Zyw14]{zywina-14}
David Zywina.
\newblock {The Splitting of Reductions of an Abelian Variety}.
\newblock {\em International Mathematics Research Notices},
  2014(18):5042--5083, 2014.

\bibitem[Zyw20]{zywina-20}
David Zywina.
\newblock Determining monodromy groups of abelian varieties.
\newblock {\em preprint}, 2020.
\newblock {\tt arXiv:2009.07441v1 }.

\end{thebibliography}

\end{document}